\documentclass[a4paper,11pt,reqno]{amsart}

\usepackage[utf8]{inputenc}

\usepackage[T1]{fontenc}
\usepackage{lmodern}

\usepackage[colorinlistoftodos,bordercolor=orange,backgroundcolor=orange!20,linecolor=orange,textsize=normalsize]{todonotes}

\usepackage{amsmath}  
\usepackage{amssymb}     
\usepackage{bbm}

\usepackage{bm}
\usepackage{hyperref}
\usepackage{mathrsfs}
\usepackage{mathtools} 
\usepackage{paralist}
\usepackage{fixmath}
\usepackage{interval}

\usepackage[capitalize]{cleveref}

\Crefname{figure}{Figure}{Figures}

\crefformat{equation}{\textup{#2(#1)#3}}
\crefrangeformat{equation}{\textup{#3(#1)#4--#5(#2)#6}}
\crefmultiformat{equation}{\textup{#2(#1)#3}}{ and \textup{#2(#1)#3}}
{, \textup{#2(#1)#3}}{, and \textup{#2(#1)#3}}
\crefrangemultiformat{equation}{\textup{#3(#1)#4--#5(#2)#6}}%
{ and \textup{#3(#1)#4--#5(#2)#6}}{, \textup{#3(#1)#4--#5(#2)#6}}{, and \textup{#3(#1)#4--#5(#2)#6}}

\Crefformat{equation}{#2Equation~\textup{(#1)}#3}
\Crefrangeformat{equation}{Equations~\textup{#3(#1)#4--#5(#2)#6}}
\Crefmultiformat{equation}{Equations~\textup{#2(#1)#3}}{ and \textup{#2(#1)#3}}
{, \textup{#2(#1)#3}}{, and \textup{#2(#1)#3}}
\Crefrangemultiformat{equation}{Equations~\textup{#3(#1)#4--#5(#2)#6}}%
{ and \textup{#3(#1)#4--#5(#2)#6}}{, \textup{#3(#1)#4--#5(#2)#6}}{, and \textup{#3(#1)#4--#5(#2)#6}}

\usepackage{tikz}

\usetikzlibrary{arrows}
\usetikzlibrary{patterns}
\usetikzlibrary{calc}
\usetikzlibrary{shapes}
\usetikzlibrary{positioning}

\usepackage[scr=boondox,scrscaled=1.05]{mathalfa}

\renewcommand{\geq}{\geqslant}
\renewcommand{\leq}{\leqslant}

\renewcommand{\le}{\leq}
\renewcommand{\ge}{\geq}

\newcommand{\arxiv}[1]{\href{http://arxiv.org/abs/#1}{arXiv:#1}}

\DeclareMathAlphabet{\mathbfcal}{OMS}{cmsy}{b}{n}
\DeclareMathAlphabet{\mathbbold}{U}{bbold}{m}{n}

\newcommand{\E}{\mathbb{E}}
\newcommand{\Z}{\mathbb{Z}}
\newcommand{\N}{\mathbb{N}} 
\newcommand{\R}{\mathbb{R}}

\newcommand{\trop}[1][]{\ifthenelse{\equal{#1}{}}{ \mathbb{T} }{ \mathbb{T}(#1) }}

\newcommand{\card}[1]{|{#1}|}

\newtheorem{theorem}{Theorem}[section]
\newtheorem{proposition}[theorem]{Proposition}

\newtheorem{corollary}[theorem]{Corollary}

\newtheorem{lemma}[theorem]{Lemma}

\theoremstyle{definition}
\newtheorem{definition}[theorem]{Definition}

\theoremstyle{remark}
\newtheorem{remark}[theorem]{Remark}
\newtheorem{example}[theorem]{Example}

\tikzset{grid/.style={gray!30,very thin}}
\tikzset{axis/.style={gray!50,->,>=stealth'}}
\tikzset{convex/.style={draw=none,fill=lightgray,fill opacity=0.7}}
\tikzset{convexborder/.style={very thick}}
\tikzset{point/.style={blue!50}}
\tikzset{hs/.style={fill opacity=0.3,fill=orange,draw=none}}
\tikzset{hsborder/.style={orange,ultra thick,dashdotted}}

\newcommand{\overbar}[1]{\mkern 1.5mu\overline{\mkern-1.5mu#1\mkern-1.5mu}\mkern 1.5mu}

\newcommand{\dgraph}{\vec{\mathcal{G}}}

\newcommand{\vertices}{V}
\newcommand{\vertex}{v}
\newcommand{\edges}{E}
\newcommand{\edge}{e}

\newcommand{\payoff}{r}

\newcommand{\states}{V}
\newcommand{\state}{v}
\newcommand{\stateII}{w}
\newcommand{\stateIII}{w'}
\newcommand{\sea}{W}
\newcommand{\sources}{S}
\newcommand{\rvar}{X}
\newcommand{\rvarII}{\tau}

\newcommand{\transition}{P}

\newcommand{\visits}[3]{{\zeta_{#2#3}^{#1}}}
\newcommand{\ind}{\mathbf{1}}
\newcommand{\weight}{p}
\newcommand{\Prob}{\mathbb{P}}
\newcommand{\Id}{I}

\newcommand{\dpath}[3][]{(#1\ifthenelse{\equal{#1}{}}{}{,} #2 \rightarrow #3   )}
\newcommand{\ndpath}[3][]{(#1\ifthenelse{\equal{#1}{}}{}{,} #2 \nrightarrow #3   )}

\newcommand{\reclass}{C}

\newcommand{\recstates}{S}
\newcommand{\absor}{\psi}

\newcommand{\subedges}{\edges'}

\newcommand{\roots}{R}
\newcommand{\forests}{\mathcal{F}}
\newcommand{\stdist}{\pi}
\newcommand{\comden}{M}

\newcommand{\nat}{m}

\newcommand{\bias}{u}
\newcommand{\gameval}{\chi}
\newcommand{\linf}[1]{\|{#1}\|_{\infty}}
\newcommand{\mchain}{\mathcal{X}}
\newcommand{\denprod}{D}

\title[Bit-sizes of stationary distributions in finite Markov chains]{Optimal bounds for bit-sizes of stationary distributions in finite Markov chains}

\author[Mateusz Skomra]{Mateusz Skomra}
\address{LAAS-CNRS, Universit{\'e} de Toulouse, CNRS, Toulouse, France}
\email{mateusz.skomra@laas.fr}
 
\begin{document}

\begin{abstract}
An irreducible stochastic matrix with rational entries has a stationary distribution given by a vector of rational numbers. We give an upper bound on the lowest common denominator of the entries of this vector. Bounds of this kind are used to study the complexity of algorithms for solving stochastic mean payoff games. They are usually derived using the Hadamard inequality, but this leads to suboptimal results. We replace the Hadamard inequality with the Markov chain tree formula in order to obtain optimal bounds. We also adapt our approach to obtain bounds on the absorption probabilities of finite Markov chains and on the gains and bias vectors of Markov chains with rewards.
\end{abstract}

\maketitle

\section{Introduction}

In this note, we study the following problem. Suppose that $\transition \in [0,1]^{n \times n}$ is an irreducible stochastic matrix whose entries are rational numbers with a common denominator $\comden \in \N$. Then, the stationary distribution $\stdist \in \interval[open left]{0}{1}^{n}$ of $\transition$ is a vector with rational entries. Our aim is to obtain an optimal upper bound on the lowest common denominator of the numbers $(\stdist_{i})_{i = 1}^{n}$, which bounds the number of bits needed to encode $\stdist$.

\subsection{Context and motivation}
Our main motivation to study the problem stated above comes from the area of stochastic mean payoff games, which form a generalization of finite Markov decision processes. A stochastic mean payoff game is a zero-sum game played by two players (Min and Max) who move a token along the edges of a finite directed graph $\dgraph = (\states, \edges)$. Some vertices of the graph are controlled by player Min, some are controlled by player Max, and some are controlled by nature, which moves the token according to some fixed probability distribution. Furthermore, each vertex $\state$ of the graph is equipped with an integer payoff $\payoff_{\state} \in \Z$. The players are supposed to play according to positional strategies, i.e., their decisions depend only on the current position of the token. In particular, if the token lands twice on the same vertex controlled by one of the two players, then this player makes the same decision on both occasions. As a consequence, once the strategies of the players are fixed, the movement of the token is described by a Markov chain $(X_0, X_1, \dots)$ on $\states$, where $X_0$ is the starting position of the token, and the randomness of this process comes only from the random decisions made by the nature. The payoff of player Max is given by the average reward criterion
\[
\lim_{N \to \infty} \frac{1}{N}\E(\payoff_{X_0} + \dots + \payoff_{X_N}) \, .
\]
Player Max aims to maximize this quantity, while player Min wants to minimize it. It is known that stochastic mean payoff games have always have optimal strategies \cite{liggett_lippman}. In other words, there exists a vector $\gameval \in \R^{\states}$, known as the \emph{value} of the game, such that player Max has a strategy that guarantees that the payoff is not smaller than $\gameval_\state$ for all initial states $\state$. Likewise, player Min has a strategy that guarantees that the payoff is not greater than $\gameval_\state$. In particular, if both player play optimally, then the final payoff is equal to $\gameval_\state$. Stochastic mean payoff games attracted a significant interest in the computer science literature thanks to their uncertain complexity status. Even though optimal strategies exist, finding them algorithmically is a nontrivial task. In particular, it is not known if these strategies can be found in polynomial time and this problem has been open for 30 years, even in some restrictive cases (deterministic mean payoff games and parity games) \cite{condon,gurvich,emerson_jutla}. We refer the reader to \cite{andersson_miltersen,zwick_paterson,comin_rizzi,ibsen-jensen_miltersen,boros_gurvich_makino,halman,calude_parity_games,parity_lower_bounds,filar_vrieze} for more information about mean payoff games and the related algorithmic issues. We also note that the one-player variant of these games is equivalent to Markov decision processes with average reward criterion, studied for instance in \cite{puterman}.

Numerous algorithms for solving stochastic mean payoff games that are proposed in the literature, such as the value iteration algorithms or the pumping algorithm, approximate the value $\gameval$ without knowing the optimal strategies of the game. When the value is approximated to a sufficient precision, a rounding procedure is used to find $\gameval$ exactly. We refer to \cite{condon,ibsen-jensen_miltersen,auger_strozecki,boros_gurvich_makino} for examples of such algorithms. In order to use a rounding procedure, one needs to have a bound on the precision needed to recover $\gameval$. This is done by bounding the denominators of $\gameval$. Such a bound can be obtained using the Hadamard inequality, and this approach is used in \cite{condon,auger_strozecki,generic_uniqueness,boros_gurvich_makino}, but it leads to suboptimal results in many cases of interest. In this note, we propose to use a more combinatorial approach, based on the Markov chain tree formula~\cite[Lemma~3.2]{catoni_rare_transitions}, to obtain optimized bounds. As noted above, $\gameval$ is the payoff of player Max obtained when both players play optimally. Let $(X_0, X_1, \dots)$ be the Markov chain obtained under the optimal strategies and let $\transition$ be its transition matrix. Our basic case of interest arises when $\transition$ is irreducible. Then, \cite[Appendix~A.4]{puterman} shows that $\gameval$ does not depend on the initial state, $\gameval = \eta(1,1,\dots,1)$, and $\eta = \payoff^{T}\stdist$, where $\stdist$ is the stationary distribution of $\transition$. Thus, the denominator of $\eta$ is not greater than the lowest common denominator of $(\stdist_{\state})_{\state \in \states}$, which leads to the problem stated in the first paragraph of this note.

\subsection{Main results}

Throughout this note, we use the following notation. We denote $[m] = \{1,\dots,m\}$ for $m \in \N$. Furthermore, let $\transition \in [0,1]^n$ be a stochastic matrix with rational entries and let $\mchain \coloneqq (X_0, X_1, \dots)$ be a Markov chain on the state space $\states \coloneqq [n]$ with transition matrix $\transition$. In general, we do not suppose that $\transition$ is irreducible, since most of our results do not require this assumption. For every $i \in [n]$, let $\comden_i$ be the lowest common denominator of the entries in the $i$th row of $\transition$, and let $\comden$ be the lowest common denominator of all the entries of $\transition$. We also put
\[
\denprod \coloneqq \comden_1\comden_2\dots\comden_n \, .
\]
Moreover, we denote by $\reclass_1, \dots, \reclass_p \subset [n]$ the recurrent classes of $\mchain$ and, for all $\ell \in [p]$, we denote by $\stdist^{(\ell)} \in \interval[open left]{0}{1}^{\reclass_{\ell}}$ the stationary distribution on $\reclass_{\ell}$. Furthermore, let $\payoff \in \Z^n$ be a vector of integer numbers and let $\gameval \in \R^n$ be the \emph{gain} vector defined as
\[
\forall i, \, \gameval_i \coloneqq  \lim_{N \to \infty} \frac{1}{N}\E(\payoff_{X_0} + \dots + \payoff_{X_N} \mid X_0 = i) \, .
\]
We note that $\gameval$ is well defined and given by
\begin{equation}\label{eq:value}
\forall i, \, \gameval_i = \sum_{\ell = 1}^{p} \absor(i, \reclass_\ell)\eta^{(\ell)} \, ,
\end{equation}
where $\absor(i, \reclass_\ell)$ denotes the probability that the Markov chain starting at $i$ reaches $\reclass_\ell$, and $\eta^{(\ell)} \coloneqq \sum_{j \in \reclass_\ell}\payoff_{j}\stdist^{(\ell)}_j$ for all $\ell \in [p]$. The formula \cref{eq:value} follows from the ergodic theorem of finite Markov chains, see \cite[Appendix~A.4]{puterman} and \cite[Part I, \S 6--\S 9]{chung_markov_chains} for detailed information. Our main result is the following theorem for irreducible matrices and its corollary, which holds even if $\transition$ is not irreducible.

\begin{theorem}\label{thm:stationary}
Suppose that $\transition$ is irreducible and let $\stdist \in \interval[open left]{0}{1}^{n}$ denote its stationary distribution. Then, $\stdist$ is a vector of rational numbers whose lowest common denominator is not greater than $\min\{n\denprod, n\comden^{n-1}\}$.
\end{theorem}
\begin{corollary}\label{cor:value_ergodic}
Suppose that $\gameval = \eta(1,1,\dots,1)$ for some $\eta \in \R$. Then, $\eta$ is a rational number with denominator not greater than $\min\{n\denprod, n\comden^{n-1}\}$.
\end{corollary}

Before discussing these results, observe that $\comden \le \denprod \le \comden^k$, where $k$ is the number of rows of $\transition$ that have at least two nonzero entries, $k \coloneqq \{i \in [n] \colon \exists j, \, 0 < \transition_{ij} < 1\}$. In particular, we have the inequality 
\begin{equation}\label{eq:deter_states}
\min\{n\denprod, n\comden^{n-1}\} \le n\comden^{\min\{k, n-1\}} \, .
\end{equation}

As noted above, our proof of \cref{thm:stationary} relies on a combinatorial formula for stationary distributions, known as the Markov chain tree formula \cite[Lemma~3.2]{catoni_rare_transitions}. By comparison, \cite[Lemma~4.10]{generic_uniqueness} uses the Hadamard inequality to obtain a bound $n^{n/2}\comden^{n}$ for the same problem. A more precise application of the Hadamard inequality is used in \cite[Lemma~6]{boros_gurvich_makino} to obtain a bound of the form $kn(2\comden)^{k+1}$. The inequality \cref{eq:deter_states} shows that our estimate is better than both of these bounds. Even more, in \cref{pr:st_optimality} we show that our bound is essentially optimal, in the sense that it cannot be improved even by a multiplicative constant. The proof of \cref{pr:st_optimality} also shows that this bound remains optimal even if we only want to bound the denominators of $\stdist_i$ separately. We also note that the interest of having bounds that depend on $k$ is that these types of bounds may be used for stochastic mean payoff games with bounded number of states controlled by nature, see \cite{gimbert_horn,ibsen-jensen_miltersen,auger_strozecki,boros_gurvich_makino} for more discussion. Furthermore, we point out that \cref{eq:value} shows that the assumption of \cref{cor:value_ergodic} is satisfied if $\mchain$ is irreducible or has only one recurrent class, but it may also be satisfied even if $\mchain$ has multiple recurrent classes. We refer to \cite{ergodicity_conditions,boros_gurvich_makino} for conditions that ensure that a stochastic mean payoff game has a value that does not depend on the initial state.

Our next result is especially useful in the situation in which $\mchain$ has an absorbing state, $\transition_{jj} = 1$, and we put $\payoff_j \coloneqq 1$ and $\payoff_i \coloneqq 0$ for all $i \neq j$. In this case, \cref{eq:value} shows that the gain $\gameval_i$ is equal to the probability that the Markov chain starting at $i$ reaches $j$. This situation arises in simple stochastic games~\cite{condon} and some of its generalizations~\cite{gimbert_horn,auger_strozecki,auger_montjoye_strozecki}. In particular, in order to bound the denominator of $\gameval_i$, we want to bound the denominators of absorption probabilities (which are rational numbers). To do so, we use an adaptation of the Markov chain tree formula to absorption probabilities, following the approach presented in~\cite{catoni_rare_transitions}. In order to state out estimate, let $T \subset [n]$ denote the set of transient states of the Markov chain $\mchain$ and put $\denprod_T \coloneqq \prod_{i \in T}\comden_i$. Our approach gives the following result.

\begin{theorem}\label{thm:absorption}
The numbers $\bigl(\absor(i, \reclass_{\ell})\bigr)_{i \in [n], \ell \in [p]}$ are rational and their lowest 
common denominator is not greater that $\min\{\denprod_T, \comden^{n-2}\}$.
\end{theorem}

As previously, we have the inequality $\denprod_T \le \comden^{k_T}$, where $k_{T} \le k$ denotes the number of rows of $\transition$ that have at least two nonzero entries and represent transient states of the Markov chain, i.e.,
\[
k_T = \{i \in [n] \colon i \text{ is transient and there exists $j \in [n]$ such that } 0 < \transition_{ij} < 1 \} \, .
\]
In particular, we have
\begin{equation}\label{eq:trans_states}
\min\{\denprod_T, \comden^{n-2}\} \le \comden^{\min\{k_T,n-2\}} \,
\end{equation}
and this bound improves the bounds obtained in the literature using the Hadamard inequality~\cite{condon,auger_strozecki}. Furthermore, this bound is tight as shown in \cref{ex:absorption}. Combining \cref{cor:value_ergodic,thm:absorption,eq:value}, we obtain the following estimate on the gain vector in general chains.

\begin{corollary}\label{cor:value_general}
The numbers $(\gameval_i)_{i \in [n]}$ are rational and their lowest common denominator is not greater than $3^{s/2}\denprod$, where $s$ denotes the number of recurrent states of a Markov chain with transition matrix $\transition$.
\end{corollary}

The main difference between the bounds of \cref{cor:value_ergodic} and \cref{cor:value_general} is that the latter bound is exponential in $s$. The example presented in \cite{boros_gurvich_discounted_approximations} shows that this is unavoidable in general chains even if $k = 1$.

To state our final result, we recall the notion of a bias vector. Given $\transition$ and $\payoff$, \cite[Theorem~8.2.6]{puterman} shows that the gain $\gameval$ can be found by solving the system of equalities
\begin{equation}\label{eq:bias}
\begin{cases}
\transition \gameval &= \gameval \\
\transition \bias &=  \gameval + \bias - \payoff
\end{cases}
\end{equation}
in variables $(\gameval, \bias) \in \R^{2n}$. More precisely, \cref{eq:bias} has a solution and any such solution $(\gameval', \bias)$ satisfies $\gameval' = \gameval$. If $(\gameval, \bias)$ is a solution of \cref{eq:bias}, then we say that $\bias$ is a \emph{bias} vector. In general, a bias vector is not unique, even up to an additive constant. Bias vectors play an important role in the policy iteration algorithms for Markov decision processes~\cite[Chapter~9]{puterman} and for stochastic mean payoff games~\cite{cdcdetournay}. Moreover, in a recent work, Allamigeon, Gaubert, Katz, and Skomra proposed a condition number for stochastic mean payoff games that governs the complexity of the value iteration algorithm~\cite{mtns2018}. This condition number depends on the quantity $\inf_{\bias}\|\bias\|_{H}$, where the infimum goes over all bias vectors of the Shapley operator associated with a stochastic mean payoff game, and $\| \cdot \|_{H}$ denotes the \emph{Hilbert seminorm}, $\|\bias\|_{H} \coloneqq \max_i \bias_i - \min_i \bias_i$. Since $\|\bias\|_{H} \le 2\linf{\bias}$, one can use the supremum norm to bound $\inf_{\bias}\|\bias\|_{H}$. The complexity estimates on value iteration obtained in \cite{mtns2018} rely on \cref{thm:stationary} and on the following result.

\begin{theorem}\label{thm:bias}
Suppose that $\bias \in \R^n$ is a bias vector of $(\payoff, \transition)$. We have the following estimates:
\begin{enumerate}[i)]
\item if for every $\ell \in [p]$ there exists $i_{\ell} \in \reclass_{\ell}$ such that $\bias_{i_{\ell}} = 0$, then $\linf{\bias} \le 2\linf{r}n\min\{\denprod, \comden^{n-1}\}$;
\item if $\sum_{i \in \reclass_{\ell}}\bias_i\stdist^{(\ell)}_i = 0$ for all $\ell \in [p]$, then $\linf{\bias} \le 4\linf{r}n\min\{\denprod, \comden^{n-1}\}$.
\end{enumerate}
\end{theorem}

The discussion in \cite[Section~8.2.3]{puterman} implies that bias vectors of both kinds exist for any pair $(\payoff, \transition)$. Furthermore, these types of bias vectors are particularly useful in the policy iteration algorithms for Markov decision processes~\cite[Section~9.2]{puterman}.

A preliminary version of the results presented in this note appeared in the PhD thesis of the author~\cite[Chapter~8]{skomra_phd}. We note that \cref{thm:absorption}, which was only briefly mentioned in \cite[Remark~8.46]{skomra_phd}, has since been obtained independently by Auger, Badin de Montjoye, and Strozecki~\cite[Theorem~23]{auger_montjoye_strozecki}, using a similar technique (the proof in \cite{auger_montjoye_strozecki} is based on the matrix tree theorem).

\subsection{Organization of the paper} The rest of the paper is organized as follows. In \cref{sec:trees} we present the necessary notions on directed trees and forests, which are used in the Markov chain tree formula. In \cref{sec:markov} we present this formula and its adaptation to absorption probabilities. \Cref{sec:main} contains the proofs of \cref{thm:stationary,thm:absorption} and their corollaries. Finally, we present the proof of \cref{thm:bias} in \cref{sec:bias}.

\section{Preliminaries}

\subsection{Rooted trees and forests}\label{sec:trees}
Let $\dgraph = (\vertices, \edges)$ be a directed graph. In this paper, we allow a directed graph to have loops, but not multiple edges, i.e., we suppose that $\edges$ is a subset of $\{(u,v) \colon u,v \in \vertices\}$. If $\subedges \subset \edges$ is a subset of edges, then we denote by $\dgraph(\subedges) = (\vertices, \subedges)$ the subgraph that consists of all the vertices of $\dgraph$, but the edges taken only from $\subedges$.

\begin{definition}\label{def:forest}
Let $\subedges \subset \edges$. We say that the graph $\dgraph(\subedges)$ is a \emph{rooted forest} if it does not have any directed cycles and every vertex of $\dgraph(\subedges)$ has at most one outgoing edge. We say that a vertex $\vertex \in \vertices$ is a \emph{root} of $\dgraph(\subedges)$ if it has no outgoing edges. We say that $\dgraph(\subedges)$ is a \emph{rooted tree} if it is a rooted forest and has exactly one root.
\end{definition}

(In \cref{def:forest} we use the convention that a loop is a directed cycle, so that a rooted forest does not have any loops.) \Cref{fig:forest} depicts a rooted forest. Before presenting the relationship between the rooted forests and Markov chains, let us give a few comments about \cref{def:forest}. First, we point out that a rooted forest is, indeed, a forest, i.e., it does not contain any undirected cycle. This follows from the fact that any undirected cycle that does not come from a directed cycle contains a vertex with two outgoing edges. Second, we note that the number of connected components of a rooted forest is equal to the number of its roots. This follows from the fact that from every vertex there is a unique directed path leading to a root. In particular, a rooted tree has one connected component, i.e., it is a tree. The following remark explains the vocabulary and conventions used in \cref{def:forest}.

\begin{figure}[t]
\begin{center}
\centering
\begin{tikzpicture}[scale=0.9,>=stealth',row/.style={draw,circle,minimum size=1.2cm},col/.style={draw,rectangle,minimum size=0.5cm},av/.style={draw, circle,fill, inner sep = 0pt,minimum size = 0.2cm}]

\node[row] (i1) at (3, 4) {$1$};
\node[row] (i2) at (9,4) {$2$};

\node[row] (i3) at (2,2) {$3$};
\node[row] (i4) at (9,2) {$5$};

\node[row] (i5) at (10,0) {$10$};
\node[row] (i6) at (8,0) {$9$};

\node[row] (i7) at (1,0) {$6$};
\node[row] (i8) at (4,2) {$4$};

\node[row] (i9) at (3,0) {$7$};
\node[row] (i10) at (5,0) {$8$};

\draw[->] (i4) to (i2);
\draw[->] (i3) to (i1);
\draw[->] (i5) to (i4);
\draw[->] (i6) to (i4);
\draw[->] (i7) to (i3);
\draw[->] (i8) to (i1);
\draw[->] (i9) to (i8);
\draw[->] (i10) to (i8);

\end{tikzpicture}
\end{center}
\caption{A rooted forest with two roots.}\label{fig:forest}
\end{figure}
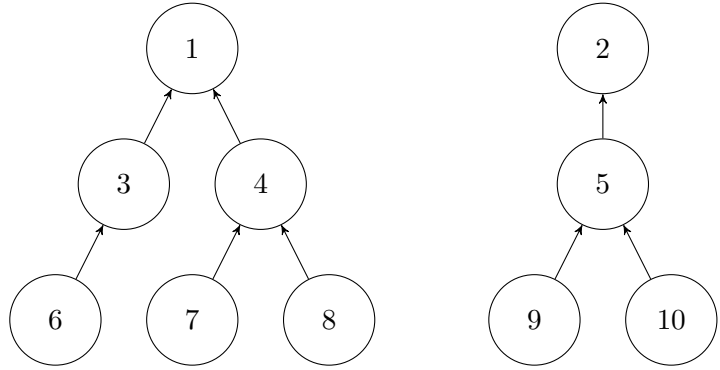

\begin{remark}
Since we suppose that $\dgraph(\subedges)$ contains all the vertices of $\dgraph$, the objects that we are considering are, in fact, \emph{spanning} rooted forests and trees. Since we never consider forests that are not spanning, we drop the word ``spanning'' from the definition. We also point out that rooted forests are also called \emph{branchings} and rooted trees are called \emph{arborescences}. Furthermore, we note that in our definition, rooted forests are oriented in such a way that every edge points ``towards the root.'' However, the opposite convention is commonly used in the literature, cf. \cite{edmonds_branchings} or \cite[Section~3.2]{schrijver_comb_opt}. The choice of orientation that we made in \cref{def:forest} is justified by the fact that the orientation ``towards the root'' corresponds to the direction of transition between states in the Markov chains that we discuss in \cref{sec:markov}.
\end{remark}

\subsection{Combinatorial formulas for Markov chains}\label{sec:markov}

As in the introduction, we denote by $\mchain \coloneqq (\rvar_{0}, \rvar_{1}, \dots)$ a Markov chain on the space $\states \coloneqq [n]$ with transition matrix $\transition$. We start by recalling the definition of an open set, see \cite[\S~3.5]{kemeny_snell}.

\begin{definition}\label{def:open}
We say that a nonempty subset of sets $\sea \subset \states$ is \emph{open} if the chain starting at any state $\state \in \sea$ can leave $\sea$ with nonzero probability,
\[
\forall \state \in \sea, \ \Prob(\exists \ell, \rvar_{\ell} \notin \sea \mid \rvar_{0} = \state) > 0 \, .
\]
\end{definition}
We note that the definition above does not imply that the chain cannot come back to $\sea$ after leaving it. In particular, $\sea$ can contain some recurrent states, but not a whole recurrent class. This also implies that the condition of \cref{def:open} can be replaced by a stronger one---if $\sea$ is open and $\state \in \sea$, then the chain starting at $\state$ will leave $\sea$ almost surely.

If $\sea \subset \states$ is open, we denote by $\rvarII_{\sea} \coloneqq \inf\{\ell \ge 1 \colon \rvar_{\ell} \notin \sea \}$ the moment when the chain leaves $\sea$ for the first time. Furthermore, for every $\state,\stateII \in \sea$ we let
\[
\visits{\sea}{\state}{\stateII}  \coloneqq \E\Bigl(\sum_{\ell = 0}^{\rvarII_{\sea} - 1} \ind_{\{\rvar_{\ell} = \stateII\} } \Big | \rvar_{0} = \state\Bigr)
\]
be the expected number of visits in $\stateII$ before leaving $\sea$, provided that the chain starts at $\state$. The following result gives a formula for computing $\visits{\sea}{\state}{\stateII}$.

\begin{lemma}[{\cite[Theorem~3.5.4(1)]{kemeny_snell}}]\label{le:inverse_transient}
Let $\hat{\transition}^{\sea \times \sea}$ denote the submatrix of $\transition$ formed by the rows and columns from $\sea$. Then, the matrix $(\Id - \hat{\transition})$ is invertible and $(\Id - \hat{\transition})^{-1}_{\state\stateII} = \visits{\sea}{\state}{\stateII}$ for all $\state,\stateII \in \sea$.
\end{lemma}

An alternative combinatorial formula for $\visits{\sea}{\state}{\stateII}$ is given by Catoni in \cite{catoni_rare_transitions}, where it is used to derive the Markov chain tree formula. In order to introduce it, we recall that a finite Markov chain is naturally associated with a directed graph representing its transitions. More precisely, we define the graph $\dgraph = (\states, \edges)$ by the condition $(\state, \stateII) \in \edges \iff \transition_{\state\stateII} > 0$. We can consider that this graph is weighted, with weights $\weight(\edge)$ that correspond to the probabilities of transitions, i.e., if $\edge = (\state, \stateII) \in \edges$, then $\weight(\edge) \coloneqq \transition_{\state\stateII}$. We also extend this definition to subsets of edges by taking the product. More precisely, if $\subedges \subset \edges$, then we define the \emph{weight} of $\subedges$ as the product of the weights of its elements,
\[
\weight(\subedges) \coloneqq \prod_{\edge \in \subedges} \weight(\edge) = \prod_{(\state, \stateII) \in \subedges} \transition_{\state\stateII} \, .
\]
The formulas that we present below consider only subsets of edges that give rise to rooted forests and trees. To this end, we denote
\begin{align*}
\forests(\roots) \coloneqq \{\subedges \subset \edges \colon &\dgraph(\subedges) \text{ is a rooted forest} \\ &\text{and its set of roots is equal to $\roots$} \} \, .
\end{align*}
It is also useful to consider the forests that contain a directed path between two fixed vertices. We denote
\begin{align*}
\forests_{\state\stateII}(\roots) \coloneqq \{\subedges \in \forests(\roots) \colon &\text{$\dgraph(\subedges)$ contains a directed path from $\state$ to $\stateII$} \} \, .
\end{align*}
We use the convention that if $\state = \stateII$, then $\forests_{\state\stateII}(\roots) \coloneqq \forests(\roots)$. If $\roots = \{\stateIII\}$ is a singleton, then we use the notation $\forests(\stateIII)$ and $\forests_{\state \stateII}(\stateIII)$ instead of $\forests(\{\stateIII\})$, $\forests_{\state \stateII}(\{\stateIII\})$. In this way, we can think of $\forests(\stateIII)$ as the set of rooted trees whose root is $\stateIII$.

\begin{lemma}[{\cite[Lemma~3.1]{catoni_rare_transitions}}]\label{visits_open}
Suppose that $\sea \subset \states$ is open and denote $\sources \coloneqq \states \setminus \sea$. Then, for every $\state, \stateII \in \sea$ we have
\[
\visits{\sea}{\state}{\stateII} = \Bigl( \sum_{\subedges \in \forests_{\state \stateII}(\sources \cup \{\stateII\})}\weight(\subedges) \Bigr) \Bigl( \sum_{\subedges \in \forests(\sources)} \weight(\subedges) \Bigr)^{-1} \, .
\]
\end{lemma}
\begin{remark}
We note that \cref{visits_open} is stated in \cite{catoni_rare_transitions} only for nontrivial subsets of irreducible chains, but the proof presented in \cite{catoni_rare_transitions} applies to arbitrary open sets of finite Markov chains.
\end{remark}

\Cref{visits_open} leads to the following two corollaries. The first one characterizes the stationary distributions of irreducible Markov chains. This corollary is known as the Markov chain tree formula, and was discovered by numerous authors~\cite{freidlin_wentzell,kohler_vollmerhaus,shubert,solberg}, see also \cite{pitman_tang} for more information. The second one characterizes the probabilities of absorption in different recurrent classes. We give the proof of the second corollary, since it is not stated in \cite{catoni_rare_transitions}.

\begin{corollary}[Markov chain tree formula, {\cite[Lemma~3.2]{catoni_rare_transitions}}]\label{st_dist_formula}
Suppose that the Markov chain is irreducible. Then, its stationary distribution $\stdist \in \interval[open left]{0}{1}^{\states}$ is given by the formula
\[
\forall \stateII \in \states, \ \stdist_{\stateII} = \Bigl( \sum_{\subedges \in \forests(\stateII)}\weight(\subedges) \Bigr) \Bigl( \sum_{\state \in \states}\sum_{\subedges \in \forests(\state)}\weight(\subedges) \Bigr)^{-1} \, .
\]
\end{corollary}

\begin{corollary}\label{absorption_formula}
Let $\recstates \subset \states$ denote the set of all recurrent states of the Markov chain and let $\reclass \subset \recstates$ be a recurrent class. Suppose that $\state \in \states$ is a transient state and let $\absor(\state, \reclass)$ be the probability that the chain starting at $\state$ reaches $\reclass$. Then, we have the equality
\[
\absor(\state, \reclass) = \Bigl( \sum_{\stateII \in \reclass} \, \sum_{\subedges \in \forests_{\state \stateII}(\recstates)}\weight(\subedges) \Bigr) \Bigl( \sum_{\subedges \in \forests(\recstates)} \weight(\subedges) \Bigr)^{-1} \, .
\]
\end{corollary}
\begin{proof}
Let $\sea \coloneqq \states \setminus \recstates$ denote the set of transient states and fix $\stateII \in \reclass$. Then, \cite[Theorem~3.5.4]{kemeny_snell} implies that the probability that the chain starting at $\state \in \sea$ goes to $\stateII$ when it leaves $\sea$ is equal to
\[
\sum_{\stateIII \in \sea}\transition_{\stateIII \stateII}\visits{\sea}{\state}{\stateIII} \, .
\]
Moreover, by \cref{visits_open} we have
\begin{align*}
\sum_{\stateIII \in \sea}\transition_{\stateIII \stateII}\visits{\sea}{\state}{\stateIII} &= \sum_{\stateIII \in \sea}\transition_{\stateIII \stateII}\Bigl( \sum_{\subedges \in \forests_{\state \stateIII}(\recstates \cup \{\stateIII\})}\weight(\subedges) \Bigr) \Bigl( \sum_{\subedges \in \forests(\recstates)} \weight(\subedges) \Bigr)^{-1} \\
&= \Bigl( \sum_{\stateIII \in \sea} \, \sum_{\subedges \in \forests_{\state \stateIII}(\recstates \cup \{\stateIII\})}\transition_{\stateIII \stateII}\weight(\subedges) \Bigr) \Bigl( \sum_{\subedges \in \forests(\recstates)} \weight(\subedges) \Bigr)^{-1} \\
&= \Bigl( \sum_{\subedges \in \forests_{\state \stateII}(\recstates)}\weight(\subedges) \Bigr) \Bigl( \sum_{\subedges \in \forests(\recstates)} \weight(\subedges) \Bigr)^{-1} \, .
\end{align*}
We obtain the claimed result by summing over $\stateII \in \reclass$.
\end{proof}

\section{Proofs of the main theorems}

In this section, we give the proofs of our main theorems. \Cref{sec:main} contains the proofs of \cref{thm:stationary,thm:absorption} and \cref{sec:bias} contains the proof of \cref{thm:bias}. 

\subsection{Stationary distributions and absorption probabilities}\label{sec:main}

The proof of our main theorem for irreducible chains relies on the following observation.

\begin{lemma}\label{le:tree_denom}
Let $Q \in \{\denprod,\comden^{n-1}\}$ and suppose that $\subedges \subset \edges$ is such that $\dgraph(\subedges)$ is a rooted tree. Then, $\weight(\subedges)Q$ is a natural number.
\end{lemma}
\begin{proof}
Since $\dgraph(\subedges)$ has exactly $n-1$ edges, $\weight(\subedges)$ is a product of $n-1$ rational numbers with common denominator $\comden$. Hence, $\weight(\subedges)\comden^{n-1}$ is a natural number. Moreover, since every vertex of $\dgraph(\subedges)$ has at most one outgoing edge, the product in $\weight(\subedges)$ involves at most one number taken from the first row of $\transition$, at most one number taken from the second row of $\transition$, at most one number taken from the third row of $\transition$ and so on. Therefore, $\weight(\subedges)\denprod$ is also a natural number.
\end{proof}

\begin{proof}[Proof of \cref{thm:stationary}]
Let $Q \in \{\denprod,\comden^{n-1}\}$. 
By \cref{st_dist_formula}, for every $\stateII \in \states$ we have
\[
\stdist_{\stateII} = \frac{Q\sum_{\subedges \in \forests(\stateII)}\weight(\subedges)}{Q\sum_{\state \in \states}\sum_{\subedges \in \forests(\state)}\weight(\subedges)} \, .
\]
By \cref{le:tree_denom}, the numerator and the denominator of the above fraction are natural numbers. In particular,
\[
Q\sum_{\state \in \states}\sum_{\subedges \in \forests(\state)}\weight(\subedges)
\]
is a common denominator $(\stdist_{\vertex})_{\vertex}$. Furthermore, note that for every $\state \in \states$ we have
\begin{equation}\label{eq:total_weight_tree}
\sum_{\subedges \in \forests(\state)}\weight(\subedges) \le \prod_{\stateII \neq \state}(\sum_{\stateIII} \transition_{\stateII \stateIII}) = 1 \, .
\end{equation}
Indeed, $\prod_{\stateII \neq \state}(\sum_{\stateIII} \transition_{\stateII \stateIII})$ is the total weight of all graphs $\dgraph(\subedges)$ in which $\state$ has no outgoing edges and every other vertex has exactly one outgoing edge. Since every tree rooted at $\state$ has these properties, we get \cref{eq:total_weight_tree}. Therefore, the lowest common denominator of $(\stdist_{\vertex})_{\vertex}$ is not greater than
\[
Q\sum_{\state \in \states}\sum_{\subedges \in \forests(\state)}\weight(\subedges) \le nQ \, . \qedhere
\]
\end{proof}

The next proposition shows that the bound of \cref{thm:stationary} is optimal for any fixed $n$. More precisely, let $g(n, \comden)$ denote the optimal bound that could be obtained in \cref{thm:stationary} for any fixed $n,\comden$, under the additional assumption that $\denprod = \comden^{n-1}$. Then, we have the following result.

\begin{figure}[t]
\begin{center}
\centering
\begin{tikzpicture}[scale=0.9,>=stealth',row/.style={draw,circle,minimum size=1.2cm},col/.style={draw,rectangle,minimum size=0.5cm},av/.style={draw, circle,fill, inner sep = 0pt,minimum size = 0.2cm}]

\node[row] (i1) at (1, 4) {$1$};
\node[row] (i2) at (3.5,4) {$2$};
\node[row] (i3) at (8,4) {$n-1$};
\node[row] (i4) at (11,4) {$n$};

\coordinate (l1) at (5.2,4);
\coordinate (l2) at (6.2,4);

\draw[->] (i1) to node[above]{$\frac{p_1}{\comden}$} (i2);
\draw[->] (i3) to node[above]{$\frac{p_{n-1}}{\comden}$} (i4);

\draw[->] (i2) to node[above]{$\frac{p_2}{\comden}$} (l1);
\draw[->] (l2) to node[above]{$\frac{p_{n-2}}{\comden}$} (i3);

\draw[->] (i4) to[out = 130, in = 50] node[above]{$1$} (i1);

\path (l1) -- node[auto=false]{\ldots} (l2);

\draw[->] (i1) to[in = -60, out = -120,looseness=7] node[below]{$1 - \frac{p_1}{\comden}$} (i1);
\draw[->] (i2) to[in = -60, out = -120,looseness=7] node[below]{$1 - \frac{p_2}{\comden}$} (i2);
\draw[->] (i3) to[in = -60, out = -120,looseness=7] node[below]{$1 - \frac{p_{n-1}}{\comden}$} (i3);

\end{tikzpicture}
\end{center}
\caption{An irreducible Markov chain from \cref{pr:st_optimality}.}\label{fig:irreducible}
\end{figure}
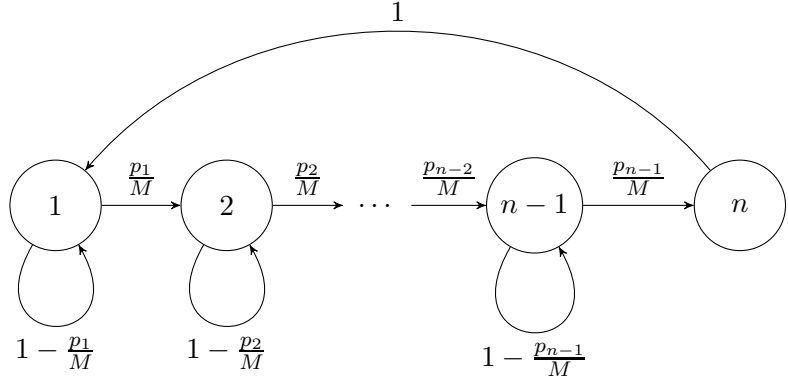

\begin{proposition}\label{pr:st_optimality}
For every $n \ge 2$ we have 
\[
\limsup_{\comden \to +\infty} \, \frac{g(n,\comden)}{n\comden^{n-1}} = 1 \, .
\]
\end{proposition}
\begin{proof}
We have $g(n,\comden) \le n\comden^{n-1}$ by \cref{thm:stationary}. To prove the opposite inequality, fix $n \ge 2$ and $\varepsilon > 0$. The following construction shows that we can find an arbitrarily large $\comden \ge 1$ and an irreducible Markov chain on $n$ states such that the lowest common denominator of $(\stdist_{\vertex})_{\vertex}$ is not smaller than $(1-\varepsilon)n\comden^{n-1}$. To do this, let $(\nat_1, \nat_2, \dots) = (2,3,5,\dots)$ be the sequence of prime numbers and let $\delta > 0$ be such that $(\frac{1}{1+\delta})^{n-1} \ge 1 - \varepsilon$. For sufficiently large $q \ge n$ we have $\nat_n \le \delta (\nat_{q}!)$. Take any such $q$ and let $\comden = \nat_q! + \nat_n$, so that $\comden \le (1+ \delta)\nat_q!$. Furthermore, let $p_1 = \nat_q! + \nat_1$, $p_2 = \nat_q! + \nat_2, \dots, p_n = \nat_q! + \nat_n = \comden$. We note that the numbers $(p_1, \dots, p_n)$ are pairwise coprime. Indeed, if a prime number $\nat$ divides $\nat_q! + \nat_i$ and $\nat_q! + \nat_j$, then it also divides $|\nat_i - \nat_j|$. Since $|\nat_i - \nat_j|$ is smaller than $\nat_n \le \nat_q$, the number $\nat$ also divides $\nat_q!$ and so it divides both $\nat_i$ and $\nat_j$, which is a contradiction. Consider the irreducible Markov chain shown in \cref{fig:irreducible} and note that this chain satisfies $\denprod = \comden^{n-1}$. Moreover, in this chain every state $\state \in [n]$ is a root of exactly one rooted tree and the weight of this tree is equal to $\prod_{\stateII \neq \state}\frac{p_{\stateII}}{\comden}$. Hence, if we denote $Q_{\state} = p_1\dots p_n / p_\state$ for all $\state$, then \cref{st_dist_formula} shows that the stationary distribution of this chain is given by 
\begin{equation}\label{ex:irred_example}
\forall \state, \, \stdist_{\state} = \frac{Q_{\state}}{\sum_{\stateII \in \states} Q_{\stateII}} \, .
\end{equation}
We note that the numbers $Q_{\state}$ are natural. Even more, the fraction in \cref{ex:irred_example} is simple. Indeed, if a prime number $\nat$ divides $Q_\state$, then it divides exactly one of $p_1, \dots, p_n$ because these numbers are pairwise coprime. Suppose that $\nat$ divides $p_i$. Then, $\nat$ divides $Q_{\stateII}$ for all $\stateII \neq i$ and it does not divide $Q_{i}$. In particular, $m$ does not divide $\sum_{\stateII \in \states} Q_{\stateII}$, showing that the fraction in \cref{ex:irred_example} cannot be simplified. Therefore, the lowest common denominator of $(\stdist_{\vertex})_{\vertex}$ is equal to $\sum_{\stateII \in \states} Q_{\state}$. Furthermore, we have
\[
\sum_{\stateII \in \states} Q_{\stateII} \ge np_1 \dots p_{n-1} \ge n (\nat_q!)^{n-1} \ge n(\frac{1}{1+\delta})^{n-1}\comden^{n-1} \ge (1- \varepsilon)n\comden^{n-1} \, . \qedhere
\]
\end{proof}

\begin{remark}
The fact that the fractions in \cref{ex:irred_example} are simple implies that the bound of \cref{thm:stationary} remains optimal even if one only wants to bound the denominators of $\stdist_\state$ separately.
\end{remark}

The theorem for absorption probabilities follows by a similar argument. The following lemma is analogous to \cref{le:tree_denom}.

\begin{lemma}\label{le:forest_denom}
Let $Q \in \{\denprod_T, \comden^{n-2}\}$ and let $\recstates \subset \states$ denote the set of recurrent states of $\mchain$. Suppose that $\card{\recstates} \ge 2$ and that $\subedges \subset \edges$ is such that $\dgraph(\subedges)$ is a forrest rooted at $\recstates$. Then, $\weight(\subedges)Q$ is a natural number.
\end{lemma}
\begin{proof}
Since $\card{\recstates} \ge 2$, the graph $\dgraph(\subedges)$ has at least two roots, and so $\weight(\subedges)$ is a product of at most $n-2$ rational numbers with common denominator $\comden$. Thus, $\weight(\subedges)\comden^{n-2}$ is a natural number. Furthermore, $\weight(\subedges)$ is a product obtained by taking one number from each row of $\transition$ that corresponds to a transient state of $\mchain$. Therefore, $\weight(\subedges)\denprod_T$ is also a natural number.
\end{proof}

\begin{proof}[Proof of \cref{thm:absorption}]
Let $\recstates \subset \states$ denote the set of recurrent states of $\mchain$. If $\mchain$ has exactly one recurrent state $\stateII \in \states$, then $\absor(\state, \{\stateII\}) = 1$ for all $\state \neq \stateII$ and the claim is trivial. From now on we suppose that $\card{\recstates} \ge 2$. Let $Q \in \{\denprod_T, \comden^{n-2}\}$. Then, \cref{absorption_formula} shows that for every $\state \in \states$ we have
\[
\absor(\state, \reclass) = \frac{Q\sum_{\stateII \in \reclass} \sum_{\subedges \in \forests_{\state \stateII}(\recstates)}\weight(\subedges)}{Q\sum_{\subedges \in \forests(\recstates)} \weight(\subedges)} \, .
\]
By \cref{le:forest_denom}, both the numerator and the denominator of the fraction above are natural numbers. Furthermore, we have
\[
\sum_{\subedges \in \forests(\recstates)}\weight(\subedges) \le \prod_{\stateII \notin \recstates}(\sum_{\stateIII} \transition_{\stateII \stateIII}) = 1 \, .
\]
Hence, the lowest common denominator of $\bigl(\absor(\state, \reclass)\bigr)_{\state \in \states, \reclass \in \mathcal{C}}$ is not greater than
\[
Q\sum_{\subedges \in \forests(\recstates)} \weight(\subedges) \le Q \, . \qedhere
\]
\end{proof}

The bound of \cref{thm:absorption} is attained for every value of $n, \comden$ as shown by the next example.

\begin{figure}[t]
\begin{center}
\centering
\begin{tikzpicture}[scale=0.9,>=stealth',row/.style={draw,circle,minimum size=1.2cm},col/.style={draw,rectangle,minimum size=0.5cm},av/.style={draw, circle,fill, inner sep = 0pt,minimum size = 0.2cm}]

\node[row] (i1) at (1, 4) {$1$};
\node[row] (i2) at (3.5,4) {$2$};
\node[row] (i3) at (8,4) {$n-2$};
\node[row] (i4) at (11,4) {$n$};

\node[row] (i5) at (6,7) {$n-1$};

\coordinate (l1) at (5.2,4);
\coordinate (l2) at (6.2,4);

\draw[->] (i1) to node[above]{$\frac{1}{\comden}$} (i2);
\draw[->] (i3) to node[above]{$\frac{1}{\comden}$} (i4);

\draw[->] (i2) to node[above]{$\frac{1}{\comden}$} (l1);
\draw[->] (l2) to node[above]{$\frac{1}{\comden}$} (i3);

\draw[->] (i3) to node[above right=-1ex]{$1- \frac{1}{\comden}$} (i5);
\draw[->] (i2) to node[above left=-1ex]{$1- \frac{1}{\comden}$} (i5);
\draw[->] (i1) to[out=60, in=170] node[above left]{$1- \frac{1}{\comden}$} (i5);

\path (l1) -- node[auto=false]{\ldots} (l2);
\draw[->] (i4) to[in = 60, out = 120,looseness=5] node[above]{$1$} (i4);
\draw[->] (i5) to[in = 20, out = -20,looseness=7] node[right]{$1$} (i5);

\end{tikzpicture}
\end{center}
\caption{A Markov chain from \cref{ex:absorption}.}\label{fig:absorption}
\end{figure}
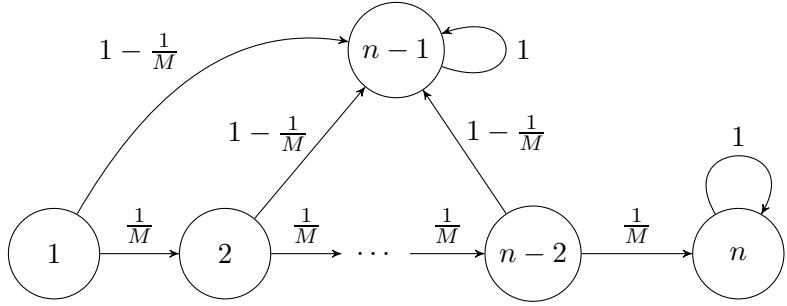

\begin{example}\label{ex:absorption}
Consider the Markov chain depicted in \cref{fig:absorption}. It is clear that we have $\denprod_T = \comden^{n-2}$ and $\absor(1,n) = 1/\comden^{n-2}$, showing that the bound of \cref{thm:absorption} is attained.
\end{example}

We now present the proofs of \cref{cor:value_ergodic,cor:value_general}.

\begin{proof}[Proof of \cref{cor:value_ergodic}]
Let $\state \in \states$ be a recurrent state of $\mchain$ belonging to some recurrent class $\reclass \subset \states$. Then, \cref{eq:value} shows that $\eta = \sum_{\stateII \in \reclass}\payoff_{\stateII}\stdist_{\stateII}$, where $\stdist \in \R^{\reclass}$ is the stationary distribution of $\reclass$. Hence, by applying \cref{thm:stationary} to $\reclass$ we get that $\eta$ is a rational number and that its denominator is not greater than $\min\{n\denprod,n\comden^{n-1}\}$.
\end{proof}

\begin{proof}[Proof of \cref{cor:value_general}]
Let $\reclass_1, \dots, \reclass_p$ denote the recurrent classes of $\mchain$. Furthermore, for every $\ell \in [p]$ let $\stdist^{(\ell)} \in \R^{\reclass_{\ell}}$ be the stationary distribution on $\reclass_{\ell}$ and let $\denprod_{\ell} \in \N$ be defined as $\denprod_{\ell} \coloneqq \prod_{\state \in \reclass_{\ell}}\comden_{\state}$. By combining \cref{thm:stationary,thm:absorption} we get that the lowest common denominator of the numbers $\Bigl( \bigl(\absor(\state, \reclass_{\ell})\bigr)_{\state \in \states, \ell \in [p]}, \bigl(\stdist^{(\ell)}_{\state}\bigr)_{\ell \in [p], \state \in \reclass_{\ell}}\Bigr)$ is not greater than
\[
\card{\reclass_1}\dots\card{\reclass_p}\denprod_1\dots\denprod_{p}\denprod_{T} = \card{\reclass_1}\dots\card{\reclass_p}\denprod \, .
\]
Therefore, \cref{eq:value} implies that the numbers $(\gameval_\state)_{\state \in \states}$ are rational and that their lowest common denominator is not greater than $\card{\reclass_1}\dots\card{\reclass_p}\denprod$. By the inequality of arithmetic and geometric means we get $\card{\reclass_1}\dots\card{\reclass_p} \le (s/p)^p$. Furthermore, the function $f \colon \R_{>0} \to \R$ defined as $f(p) \coloneqq p\ln(s) - p\ln(p)$ achieves its maximum when $\ln(p) = \ln(s) - 1$, i.e., $p = s/e$. Therefore, we get $(s/p)^p \le e^{s/e}$. Since $e^{1/e} < 1.5 < \sqrt{3}$, we obtain $\card{\reclass_1}\dots\card{\reclass_p} \le 3^{s/2}$, which finishes the proof.
\end{proof}

\subsection{Estimating a bias vector}\label{sec:bias}

We now give our estimates concerning bias vectors. Our proof of \cref{thm:bias} is based on the following lemma.

\begin{lemma}\label{le:estimate_visits}
Suppose that $\sea \subset \states$ is open and let $\hat{\transition}^{\sea \times \sea}$ denote the submatrix of $\transition$ formed by the rows and columns from $\sea$. Then, we have $\linf{(\Id - \hat{\transition})^{-1}} \le \prod_{\stateII \in \sea}\comden_{\stateII} \le \min\{\denprod, \comden^{n-1}\}$.
\end{lemma}
\begin{proof}
Let $\denprod_\sea \coloneqq \prod_{\stateII \in \sea}\comden_{\stateII}$. We have $\denprod_\sea \le \min\{\denprod, \comden^{n-1}\}$ because $\sea$ has at most $n-1$ states. By \cref{le:inverse_transient}, we have $(\Id - \hat{\transition})^{-1}_{\state\stateII} = \visits{\sea}{\state}{\stateII}$ for all $\state, \stateII \in \sea$. Fix $\state, \stateII$ and let $\rvarII_{\sea} = \inf\{\ell \ge 1 \colon \rvar_{\ell} \notin \sea \}$ denote the moment when the Markov chain leaves $\sea$ for the first time. Moreover, let $Z \coloneqq \sum_{\ell = 0}^{\rvarII_{\sea} - 1}\ind_{\{X_{\ell} = \stateII\}}$ denote the number of times the Markov chain visits $\stateII$ before leaving $\sea$. Under this notation, we have $\visits{\sea}{\state}{\stateII} = \E(Z | X_0 = \state)$. Furthermore, let $q \in \interval{0}{1}$ denote the probability that the Markov chain starting at $\stateII$ goes back to $\stateII$ before leaving $\sea$, i.e., $q \coloneqq \Prob(Z \ge 2 | X_0 = \stateII)$. Since $\sea$ is open, there exists a simple path in $\dgraph$ that starts in $\stateII$ and ends in some state that is outside $\sea$. The probability that the Markov chain starting from $\stateII$ follows this path is not smaller than $1/\denprod_W$. Therefore, $q \le 1 - \denprod_W^{-1}$. Furthermore, note that for all $t \ge 1$ we have $\Prob(Z \ge t | X_0 = \state) \le q^{t-1}$, because in order to achieve $Z \ge t$ the chain starting from $\state$ has to reach $\stateII$ and subsequently go back to $\stateII$ at least $t - 1$ times. Thus,
\[
\visits{\sea}{\state}{\stateII} = \E(Z | X_0 = \state) = \sum_{t = 1}^{\infty}\Prob(Z \ge t | X_0 = \state) \le \sum_{t = 1}^{\infty}q^{t-1} = \frac{1}{1 - q} \le \denprod_W. \qedhere
\]
\end{proof}

To prove \cref{thm:bias}, we start with the irreducible case and then move to the general case.

\begin{lemma}\label{le:bias_irred}
\Cref{thm:bias} is true when $\transition$ is irreducible.
\end{lemma}
\begin{proof}
We start by proving the case $i)$. Let $\bias \in \R^{\states}$ be a bias vector such that $u_n = 0$ (the proof if analogous if $\bias_\state = 0$ for some other $\state \in \states$). Denote $\sea \coloneqq \states \setminus \{n\}$ and let $\hat{\transition} \in \R^{\sea \times \sea}$ be the matrix obtained from $\transition$ by deleting the last row and column. Likewise, let $\hat{\bias},\hat{\payoff} \in \R^{\sea}$ be the vectors obtained from $\bias,\payoff$ by deleting their last coordinates. Let $\stdist \in \interval[open left]{0}{1}^{\states}$ be the stationary distribution of $\transition$ and denote $\eta = \sum_{\state \in \states} \stdist_{\state}\payoff_{\state} \in \R$, so that $\gameval = \eta(1,1,\dots,1)$. We note that $|\eta| \le \linf{\payoff}$. The definition of the bias vector and the fact that $\bias_{n} = 0$ imply the equality $\hat{\transition}\hat{\bias} = \eta + \hat{\bias} - \hat{\payoff}$. Since $\sea$ is an open set, \cref{le:inverse_transient} gives $\hat{\bias} = (\Id - \hat{\transition})^{-1}(-\eta + \hat{r})$. Hence, \cref{le:estimate_visits} shows that 
\[
\linf{\hat{\bias}} \le 2\linf{\payoff}(n-1)\min\{\denprod, \comden^{n-1}\} \le2\linf{\payoff}n\min\{\denprod, \comden^{n-1}\} \, .
\]
To prove the second case, note that the kernel of the matrix $(\Id - \transition)$ is equal to $\{\lambda(1,1,\dots,1) \colon \lambda \in \R\}$. Therefore, the set of bias vectors of $(\payoff, \transition)$ is given by $\{\lambda + \overbar{\bias} \colon \lambda \in \R\}$, where $\overbar{\bias}$ is a bias such that $\overbar{\bias}_n = 0$. Thus, all bias vectors have the same Hilbert seminorm and the previous case gives
\[
\| \bias \|_{H} = \| \overbar{\bias}\|_{H} \le 2\linf{\overbar{\bias}} \le 4\linf{\payoff}n\min\{\denprod, \comden^{n-1}\} \, .
\]
Since $\stdist^{T}\bias = 0$, for all $\state \in \states$ we get
\[
|\bias_{\state}| = |\sum_{\stateII \in \states} \stdist_{\stateII}(\bias_{\state} - \bias_{\stateII})| \le \max_{\stateII \in \states}|\bias_{\state} - \bias_{\stateII}| \le \| \bias \|_{H} \le 4\linf{\payoff}n\min\{\denprod, \comden^{n-1}\} \, . \qedhere
\]
\end{proof}

\begin{proof}[Proof of \cref{thm:bias}]
The proof of both cases follows from \cref{le:bias_irred} using the same argument, so we focus only on $i)$. Let $\recstates \coloneqq \cup_{\ell = 1}^{p} \reclass_\ell$ denote the set of recurrent states of $\mchain$ and $\sea \coloneqq \states \setminus \recstates$ denote the set of transient states. For every $\ell$, let $\transition^{(\ell)} \in \R^{\reclass_\ell \times \reclass_\ell}$ denote the submatrix of $\transition$ formed by the rows and columns with indices in $\reclass_\ell$, and let $\bias^{(\ell)},\payoff^{(\ell)} \in \R^{\reclass_\ell}$ denote the restrictions of $\bias,\payoff$ to the indices from $\reclass_\ell$. The equations \cref{eq:value,eq:bias} imply that $\bias^{(\ell)}$ is a bias of $(\payoff^{(\ell)}, \transition^{(\ell)})$. Therefore, \cref{le:bias_irred} implies that for all $\ell \in [p]$ we have $\linf{\bias^{(\ell)}} \le 2\linf{\payoff}\card{\recstates}\min\{\denprod, \comden^{n-1}\}$. Let $\overbar{\transition} \in \R^{\sea \times \recstates}$ denote the submatrix of $\transition$ formed by the rows from $\sea$ and columns from $\recstates$ and $\hat{\transition} \in \R^{\sea \times \sea}$ denote the submatrix of $\transition$ formed by the rows and columns from $\sea$. Define a vector $\overbar{\bias} \in \R^{\recstates}$ as $\forall \state \in \recstates, \overbar{\bias}_\state \coloneqq \bias^{(\ell)}_\state$, where $\ell$ is such that $\state \in \reclass_\ell$. Furthermore, let $\hat{\gameval}, \hat{\bias},\hat{\payoff} \in \R^{\sea}$ denote the vectors $\gameval, \bias, \payoff$ restricted to the coordinates from $\sea$. By the definition of the bias vector we have $\hat{\transition}\hat{\bias} + \overbar{\transition}\overbar{\bias} = \hat{\gameval} + \hat{\bias} - \hat{\payoff}$. Since $\sea$ is an open set, \cref{le:inverse_transient} gives $\hat{\bias} = (\Id - \hat{\transition})^{-1}(-\hat{\gameval} + \hat{\payoff} + \overbar{\transition}\overbar{\bias})$. Moreover, \cref{eq:value} implies that $\linf{\gameval} \le \linf{\payoff}$. Hence, by \cref{le:estimate_visits} we get
\[
\linf{(\Id - \hat{\transition})^{-1}(-\hat{\gameval} + \hat{\payoff})} \le 2\linf{\payoff}\card{\sea}\min\{\denprod, \comden^{n-1}\} \, .
\]
Furthermore, let $R \coloneqq (\Id - \hat{\transition})^{-1}\overbar{\transition} \in \R^{\sea \times \recstates}$. By \cite[Theorem~3.5.4]{kemeny_snell}, for every $\stateII \in \sea, \state \in \recstates$, $R_{\stateII \state}$ is the probability that the Markov chain starting at $\stateII$ goes to $\state$ when it leaves $\sea$. Hence $R_{\stateII \state} \ge 0$ and $\sum_{\state \in \recstates}R_{\stateII \state} = 1$ for all $\stateII \in \sea$. Therefore, we get
\[
\linf{(\Id - \hat{\transition})^{-1}\overbar{\transition}\overbar{\bias}} \le \linf{\overbar{\bias}} \le 2\linf{\payoff}\card{\recstates}\min\{\denprod, \comden^{n-1}\} \,
\]
and so $\linf{\hat{\bias}} \le 2\linf{\payoff}n\min\{\denprod, \comden^{n-1}\}$.
\end{proof}

\bibliographystyle{acm}

\end{document}